\documentclass[11pt]{amsart}
\usepackage{graphicx, amsmath, amsfonts,amsthm,mathrsfs, xcolor, verbatim, tikz, float, caption}
\newtheorem{theorem}{Theorem}

\newtheorem{proposition}[theorem]{Proposition}
\newtheorem{condition}[theorem]{Condition}
\numberwithin{theorem}{section}
\numberwithin{equation}{section}
\numberwithin{figure}{section}
\numberwithin{table}{section}

\usepackage{fullpage}

\title{An Explicit Cohen-Style Threefield Identity}
\author{Lucas Perryman-Deskins}
\thanks{This work was supported in part by NSF Grant DMS-2101906 (PI: Holly Swisher).} 
\date{\today}

\begin{document}

\begin{abstract}
In 1988, Andrews, Dyson, and Hickerson showed that a $q$-series $\sigma$ found in Ramanujan's lost notebook and related to partitions could be interpreted as counting ideals in $\mathbb{Q}(\sqrt{6})$, and found similar formulas for $\sigma$ in terms of ideals of $\mathbb{Q}(\sqrt{2})$ and $\mathbb{Q}(\sqrt{3})$.  Cohen followed this by showing more generally that for certain triples of quadratic fields, there is abelian extension and conductor so that the ray class character theta series for all three fields coincide. In the intervening years, several $q$-series counting ideals in quadratic fields have been explored, nearly all relating to the fields explored by Andrews et al.  In this paper we give an example of a ray class character theta series which stems from $\mathbb{Q}(\sqrt{-6})$, $\mathbb{Q}(i)$, and $\mathbb{Q}(\sqrt{6})$.  We use recent work of Okano to give explicit formulas via quadratic form theta series with congruence conditions stemming from each field.  We isolate an analogue of $\sigma$ for this series and show that it has coefficients also given by a partition generating function.
\end{abstract}

\maketitle

\section{Introduction}

In 1988 Andrews, Dyson, and Hickerson \cite{ADH} investigated the series 
$$\sigma(q) =\sum_{n\geq 0}\frac{q^{n\choose 2}}{(1+q)\cdots (1+q^n)} =\sum_{n\geq 0}S(n)q^n$$
listed in Ramanujan's lost notebook. In proving properties of $S(n)$, previously known to count certain partitions of $n$, they showed that $S(n)$ could also be understood as counting integral ideals of $\mathbb{Q}(\sqrt{6})$ with norm congruent to $1$ modulo $24$. Alternatively, $\sigma$ shares coefficients with an indefinite theta series for the quadratic form $x^2-6y^2$. They also introduce another series $\sigma^*(q)$, whose coefficients $2S^*(n)$ stem from an extension of a formula for $S(n)$, counting inequivalent solutions to $x^2-6y^2\equiv 1\pmod{24}$ where $x^2-6y^2<0$. Cohen \cite{Cohen} later showed this to be the same as counting integral ideals of $\mathbb{Q}(\sqrt{6})$ with norm congruent to $-1$ modulo $24$. $S^*(n)$ can also be described as a partition counting function. Combining these functions appropriately, the $q$-series
\begin{equation}\label{sigmaeqn}
    q\sigma(q^{24})+q^{-1}\sigma^*(q^{24})=\sum_{\mathfrak{a}\subset \mathbb{Z}[\sqrt{6}]}\chi_1(\mathfrak{a})q^{N(\mathfrak{a})}=\Theta_{\mathfrak{f}_1,\chi_1}
\end{equation}
is a theta series for a specific character $\chi_1$ on the $\mathbb{Q}(\sqrt{6})$ ray class group with respect to conductor $\mathfrak{f}_1=(4(3+\sqrt{6}))$.\\

Through this characterization, Andrews et al. \cite{ADH} were able to express $\sigma$ as an indefinite theta series stemming from the field $\mathbb{Q}(\sqrt{6})$, or in other words, from the quadratic form $x^2-6y^2$. In addition, they discuss similar formulas for $\sigma$ related to $\mathbb{Q}(\sqrt{2})$ and $\mathbb{Q}(\sqrt{3})$. They note that the multiple equations for $\sigma$ and $\sigma^*$ stem from an identity of indefinite theta series. That is, there exist ideals $\mathfrak{f}_2\subset \mathbb{Z}[\sqrt{2}]$ and $\mathfrak{f}_3\subset \mathbb{Z}[\sqrt{3}]$, along with characters on the corresponding ray class groups so that
\begin{equation}\label{sigmathreefield}
    \sum_{\mathfrak{a}\subset \mathbb{Z}[\sqrt{6}]}\chi_1(\mathfrak{a})q^{N(\mathfrak{a})} = \sum_{\mathfrak{a}\subset \mathbb{Z}[\sqrt{2}]}\chi_2(\mathfrak{a})q^{N(\mathfrak{a})} = \sum_{\mathfrak{a}\subset \mathbb{Z}[\sqrt{3}]}\chi_3(\mathfrak{a})q^{N(\mathfrak{a})}.
\end{equation}

Andrews et al. \cite{ADH} give further examples of $q$-series identities which can be formulated in terms of theta series with respect to pairs of quadratic fields. In an article complementary to \cite{ADH}, Cohen \cite{Cohen} gave an explanation for this apparent identity of theta series for pairs of quadratic fields, using the formulas in (\ref{sigmathreefield}) for $\sigma$ in terms of the three fields $\mathbb{Q}(\sqrt{2})$, $\mathbb{Q}(\sqrt{3})$, and $\mathbb{Q}(\sqrt{6})$ as the motivating example. He demonstrated that a similar identity of theta series occurs for any triple of quadratic fields, with the ideal and character defining the theta series stemming from a representation of a shared abelian extension, $L$. $L$ must satisfy the following condition.

\begin{condition}\label{CohenC}
$L/\mathbb{Q}$ is Galois, and $\text{Gal}(L/\mathbb{Q})$ has a cyclic center with index $4$.
\end{condition}
\begin{figure}
    \centering
\begin{tikzpicture}[scale=0.6]
\node (L) {$L$}
	child {node (B) {$B$}
		child {node (k1)  {$K_1$}}
		child {node (k2) {$K_2$}
			child {node (Q) {$\mathbb{Q}$}}
			}
		child {node (k3)  {$K_3$}}
};
\draw [-] (Q) to (k1);
\draw [-] (Q) to (k3);
\end{tikzpicture}
\caption{Threefield Identity Galois Structure}
\label{CohenF}
\end{figure}
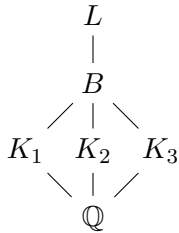
The simplest examples of Galois groups which satisfy Condition \ref{CohenC} are $D_8$, the dihedral group with $8$ elements, and $Q_8$, the quaternion group. \\

Assume $L$ satisfies Condition \ref{CohenC}. Then $L$ has three quadratic subfields, $K_1$, $K_2$, and $K_3$. For each $i\in \{1,2,3\}$, $L/K_i$ is abelian, and class field theory gives a conductor $\mathfrak{f}_i$ so that 
$$\text{Gal}(L/K_i)\cong Cl_{\mathfrak{f}_i}(K_i)/H_i$$
for a certain normal subgroup $H_i$. As described by Cohen \cite{Cohen}, for each $i$, there are two characters $\chi_{i,1}$, $\chi_{i,2}$ on $Cl_{\mathfrak{f}_i}(K_i)/H_i$ such that for any pair $1\leq i\leq 3$, $1\leq j\leq 2$, the corresponding theta series
\begin{equation}\label{Introchartheta}
\theta_{\mathfrak{f}_i,\chi_{i,j}}=\sum_{\mathfrak{a}\subset\mathcal{O}_{K_i}}\chi_{i,j}(\mathfrak{a})q^{N(\mathfrak{a})}
\end{equation}
yield the same $q$-series, thus giving an identity of the form in (\ref{sigmathreefield}). We call such an identity a {\it threefield identity}, following Mortenson \cite{Mortenson} and refer to the resulting $q$-series as {\it threefield theta series}. In the case of the identities for $\sigma$ and $\sigma^*$ given in (\ref{sigmaeqn}) and (\ref{sigmathreefield}), we have 
$$\textstyle L=\mathbb{Q}(\sqrt{2},\sqrt{3+\sqrt{3}})$$
with quadratic subfields $\mathbb{Q}(\sqrt{6})$, $\mathbb{Q}(\sqrt{2})$, and $\mathbb{Q}(\sqrt{3})$.\\

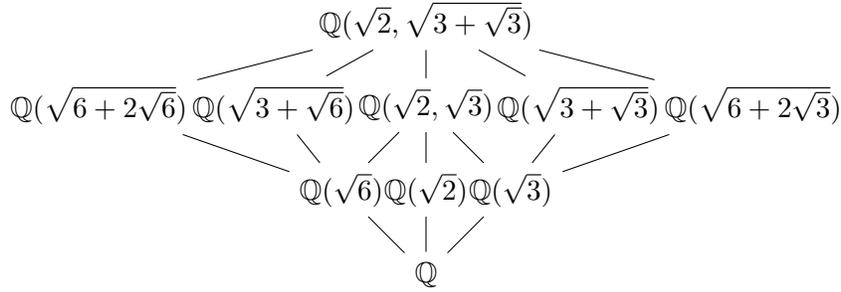
\begin{figure}\label{ADHGalois}
\begin{center}\begin{tikzpicture}[scale=0.75]
\node (L) {$\mathbb{Q}(\sqrt{2}, \sqrt{3+\sqrt{3}})$}
    child {node (A1) [xshift=-2.1cm] {$\mathbb{Q}(\sqrt{6+2\sqrt{6}})$}}
    child {node (A2) [xshift=-0.9cm] {$\mathbb{Q}(\sqrt{3+\sqrt{6}})$}}
	child {node (B) {$\mathbb{Q}(\sqrt{2},\sqrt{3})$}
		child {node (k1)  {$\mathbb{Q}(\sqrt{6})$}}
		child {node (k2) {$\mathbb{Q}(\sqrt{2})$}
			child {node (Q) {$\mathbb{Q}$}}
			}
		child {node (k3)  {$\mathbb{Q}(\sqrt{3})$}}}
    child {node (B1) [xshift=0.9cm] {$\mathbb{Q}(\sqrt{3+\sqrt{3}})$}}
    child {node (B2) [xshift=2.1cm] {$\mathbb{Q}(\sqrt{6+2\sqrt{3}})$}};
\draw [-] (Q) to (k1);
\draw [-] (Q) to (k3);
\draw [-] (A1) to (k1);
\draw [-] (A2) to (k1);
\draw [-] (B1) to (k3);
\draw [-] (B2) to (k3);
\end{tikzpicture}\end{center}
\caption{Field Extension Structure for $\sigma$ and $\sigma^*$}
\end{figure}

It is well understood that theta series associated to definite quadratic forms in $k$ variables are modular forms with weight $k/2$ on some congruence subgroup \cite{123mod}. In the $k=2$ case, one explanation for this is that the Gamma function $\Gamma(s)$ is itself the $\Gamma$-factor for the $L$-function stemming from a binary definite quadratic form. This is generally not the case for indefinite theta functions. In order to construct an automorphic form stemming from 
$$\Theta_{{(4(3+\sqrt{6})),\chi}}=\sum_{n\in \mathbb{Z}}T(n)q^n,$$
Cohen \cite{Cohen} notes that this indefinite binary theta series has an $L$-function with gamma factor $\Gamma(s/2)^2$. This leads him to draw in the Bessel function $K_0$ and construct
$$\varphi_0(x+iy)=y^{1/2}\sum_{n\in \mathbb{Z}}T(n)q^{nx/24}K_0(2\pi|n|y/24),$$
which he shows to be a Maass form. So it turns out that while the threefield theta series in (\ref{sigmaeqn}) and (\ref{sigmathreefield}) is not a modular form, its coefficients can be used to construct a Maass form.\\

Following these discoveries, others have explored a variety of $q$-series analogous to $\sigma$ and $\sigma^*$ in various ways. The earliest example is from Corson et al. \cite{Corson}, who studied a pair of $q$-series which can be expressed in terms of counting ideals in $\mathbb{Q}(\sqrt{2})$, and also have partition theoretic interpretations. Lovejoy \cite{Lovejoy1, Lovejoy3} found a variety of $q$-series related to different kinds of partition counting functions which could also be expressed in terms of counting ideals in  $\mathbb{Q}(\sqrt{6})$, $\mathbb{Q}(\sqrt{2})$, and $\mathbb{Q}(\sqrt{3})$. Lovejoy and Mallet \cite{Lovejoy4} also related partition counting functions to ideals in real quadratic fields, notably finding one identity related to $\mathbb{Q}(\sqrt{5})$. Bringmann and Kane \cite{Bringmann1} and Lovejoy and Osburn \cite{Lovejoy5} used the theory of Bailey chains \cite{andrews1986fifth, ADH} with some newly developed Bailey pairs, to discover further identities related to  $\mathbb{Q}(\sqrt{6})$, $\mathbb{Q}(\sqrt{2})$, and $\mathbb{Q}(\sqrt{3})$. Working in a different direction, Mortenson \cite{Mortenson} applied the theory from Cohen \cite{Cohen} to describe representations of primes by quadratic forms stemming from triples of quadratic fields as in Figure \ref{CohenF}. For this purpose he explicates some identities relating to ideals in the triples of fields $\mathbb{Q}(\sqrt{-2})$, $\mathbb{Q}(i)$, $\mathbb{Q}(\sqrt{2})$, and also $\mathbb{Q}(\sqrt{-5})$, $\mathbb{Q}(i)$, $\mathbb{Q}(\sqrt{5})$. A few other identities of a similar flavor have been explored, but the vast majority have stemmed from the three real quadratic fields explored in Andrews et al. \cite{ADH}.\\

One naturally asks, are there interesting $q$-series stemming from Cohen-style identities relating to other fields? The LMFDB lists over 70,000 degree $8$ Galois fields with Galois group satisfying Condition \ref{CohenC}. When looking at such degree $8$ number fields, many of the resulting threefield theta series have terms which are not clearly groupable by the congruence class of their exponents modulo some small integer. The few that do are particularly interesting. This grouping also allows us to consider the sequences of coefficients attached to each subseries, and to consider whether there is a combinatorial interpretation of these coefficients.\\

In this paper we will investigate the threefield theta series $\Theta$ corresponding to the triple of fields $\mathbb{Q}(\sqrt{-6})$, $\mathbb{Q}(i)$, and $\mathbb{Q}(\sqrt{6})$ with the latter shared in common with the example in \cite{ADH, Cohen}. The series $\Theta$ satisfied all of the above properties of interest, and is also open to additional comparison with $\sigma$ and $\sigma^*$. Like with the example in (\ref{sigmaeqn}) and (\ref{sigmathreefield}), where the exponents of the threefield theta series fall into only two classes, $1$ and $-1$, modulo $24$, the exponents of $\Theta$ also fall into only two classes, here $1$ and $5$ modulo $24$. Thus $\Theta$ naturally breaks down into two functions analogous to $\sigma$ and $\sigma^*$, and the part with exponents congruent to $1$ modulo $24$ can be compared to $\sigma$.\\

We hope to clarify some details on deriving theta series in terms of quadratic forms from those in terms of ideals in quadratic fields. For this purpose, we use results from Okano \cite{okano2021theta} detailing the correspondence between ray class theta series and theta series for binary quadratic forms with congruence conditions. Okano also gives conditions under which these theta series are weight $1$ modular forms with respect to a specific congruence subgroup.

\begin{theorem}\label{MainT}
    Let $L=\mathbb{Q}(\sqrt[4]{6},i)$. Then $L$ satisfies Condition \ref{CohenC}, and the resulting threefield theta series $\Theta$ has the form
    $$\Theta(\tau)=q\rho(q^{24})+q^5\rho^*(q^{24}),$$
    where $\Theta,$ $q\rho(q^{24}),$ and $q^5\rho^*(q^{24})$ are each weight $1$ modular forms with respect to $\Gamma_1(2304)$, and
    $$q\rho(q^{24})=\frac{\eta(48\tau)^8}{\eta(24\tau)^3\eta(96\tau)^3}.$$
\end{theorem}

\begin{proposition}\label{PartC}
    Letting 
    $$\rho(q)=\sum_{n\geq0}r(n)q^n,$$ 
    the function $r(n)$ counts the number of partitions of $n$ into distinct $3$-colored odd parts and an even number of distinct $2$-colored even parts, minus the number of partitions of $n$ into distinct $3$-colored odd parts and an odd number of distinct $2$-colored even parts.
\end{proposition}

For example, we have the following partitions of $7$ into odd parts appearing at most $3$ times, and even parts at most twice, with the number of valid colorings:
\begin{align*}
    &7 & (3) \\
    &6+1 & (6)\\
    &5+2 & (6)\\
    &5+1+1 & (9)\\
    &4+3 & (6)\\
    & 4+2+1& (12)\\
    &4+1+1+1 & (2)\\
    &3+3+1& (9)\\
    &3+2+2& (3)\\
    &3+2+1+1& (18)\\
    &2+2+1+1+1& (1).
\end{align*}
There are a total of $37$ colored partitions with an even number of even parts, and $38$ colored partitions with an odd number of even parts. Thus $r(7)=-1$.\\

The process outlined by Okano \cite{okano2021theta} to derive quadratic form theta series from ray class theta series differs depending on whether the field is real or imaginary, and on the field's narrow class group structure. In addition to including a combination of real and imaginary fields, the three fields have differing narrow class group structures:
\begin{align*}
    &Cl_{\mathbb{Q}(i)}=Cl_{\mathbb{Q}(i)}^+=\{(1)\}\\
    &Cl_{\mathbb{Q}(\sqrt{-6})}=Cl_{\mathbb{Q}(\sqrt{-6})}^+=\{(1), (2,\sqrt{-6})\}\\
    &Cl_{\mathbb{Q}(\sqrt{6})}=\{(1)\} \;\text{ and }\; Cl_{\mathbb{Q}(\sqrt{6})}^+=\{(1),(2+\sqrt{6})\}.
\end{align*}

\begin{proposition}\label{ThreeProp}
    For each field, $\mathbb{Q}(\sqrt{-6})$, $\mathbb{Q}(i)$, and $\mathbb{Q}(\sqrt{6})$, there are distinct quadratic forms and congruence conditions yielding explicit formulas for $\Theta$, $\rho$, and $\rho^*$ as sums of binary quadratic form theta series.
\end{proposition}

The rest of the paper is organized as follows. In Section 2, we review some background information from class field theory and the construction of ray class theta series. In Section 3, we prove Theorem \ref{MainT}, by applying the proof of a theorem of Okano \cite{okano2021theta} to the definition of $\Theta$ in terms of ray classes with respect to $\mathbb{Q}(\sqrt{-6})$. In Section 4, we repeat this process using the definition of $\Theta$ in terms of ray classes with respect to $\mathbb{Q}(i)$ and $\mathbb{Q}(\sqrt{6})$ yielding different expressions, and use latter to compare with formulas for $\sigma$ and $\sigma^*$ derived by Andrews et al. \cite{ADH}. In Section 5, we prove Proposition \ref{PartC}, describing the coefficients of $\rho$ with a partition counting function.

\section{Preliminaries}

\subsection{Ray Classes}

For a number field $L$, a {\it modulus} $\mathfrak{f}$ is a collection of finite and arithmetic places of $K$. Taking the ideal perspective, $\mathfrak{f}=\mathfrak{f}_0\mathfrak{f}_\infty$ can be broken into a finite part $\mathfrak{f}_0$, an ideal in $\mathcal{O}_K$ and an infinite part $\mathfrak{f}_\infty$, a subset of real embeddings of $K$. There is a group $I_\mathfrak{f}$ of fractional ideals of $K$ relatively prime to $\mathfrak{f}_0$, that is to say fractional ideals $\mathfrak{a}^{-1}\mathfrak{b}$ for $\mathfrak{a}$, $\mathfrak{b}$ integral ideals which have no prime ideal factors in common with $\mathfrak{f}_0$. $I_\mathfrak{f}$ has a normal subgroup $P_\mathfrak{f}$ made up of principal fractional ideals $(\alpha/\beta)$ with $\alpha,\beta\in \mathcal{O}_K$ satisfying $\alpha\equiv \beta \pmod{\mathfrak{f}_0}$, equivalently $\alpha-\beta\in \mathfrak{f}_0$, and $\alpha/\beta$ positive under all embeddings in $\mathfrak{f}_\infty$.
Then the {\it ray class group} of $K$ for modulus $\mathfrak{f}$ is defined 
$$Cl_K(\mathfrak{f})=I_\mathfrak{f}/P_\mathfrak{f}.$$
Two special cases are worth noting. When $\mathfrak{f}_0=\mathcal{O}_K$ and $\mathfrak{f}_\infty$ is trivial, this is the standard{ \it ideal class group} $Cl_K$. If $\mathfrak{f}_\infty$ is taken to contain all possible real embeddings of $K$, then the group is called the {\it narrow ray class group} for $\mathfrak{f}_0$, denoted $Cl_K^+(\mathfrak{f}_0)$ or just the {\it narrow class group} $Cl_K^+$ when $\mathfrak{f}_0=\mathcal{O}_K$. We will denote a ray class by $\mathfrak{C}\in Cl_K(\mathfrak{f})$, or $\mathfrak{C}=[\mathfrak{a}]$ for any $\mathfrak{a}\in \mathfrak{C}$, and represent multiplication of classes by juxtaposition.\\

Artin reciprocity states that for any abelian extension of number fields $L/K$, there exist moduli $\mathfrak{f}$ of $K$ such that for a particular normal subgroup $N_{L/K}(\mathfrak{f})\subset Cl_K(\mathfrak{f})$, there is a canonical isomorphism, the Artin-Takagi map, so that
\begin{equation}\label{ATMap}
Cl_K(\mathfrak{f})/N_{L/K}(\mathfrak{f})\xrightarrow{\sim}\text{Gal}(L/K).
\end{equation}
The least common multiple of all moduli satisfying (\ref{ATMap}) is called the {\it conductor} of the extension, which is written $\mathfrak{f}_{L/K}$.

\subsection{Ray Class Theta Series}

We use the {\it ideal norm} for integral ideals $\mathfrak{a}$, $N(\mathfrak{a})=[\mathfrak{a}:\mathcal{O}_K]$, 
the index of the ideal as an additive subgroup of the ring of integers. In particular, when $\mathfrak{a}=(\alpha)$ is principal, this corresponds to the element norm with $N(\mathfrak{a})=|N(\alpha)|$. For a fractional ideal $\mathfrak{a}^{-1}\mathfrak{b}$, this norm extends to the quotient of the norms of the integral ideals:
$N(\mathfrak{a}^{-1}\mathfrak{b})=\frac{N(\mathfrak{b})}{N(\mathfrak{a})}.$
For a modulus $\mathfrak{f}$, the norm $N(\mathfrak{f}$) is defined to be the norm of the finite part $\mathfrak{f}_0$. Given a ray class $\mathfrak{C}\in Cl_K(\mathfrak{f})$, we define the {\it ray class theta series}
\begin{equation}\label{rcldefeqn}
\theta_\mathfrak{C}=\sum_{\substack{\mathfrak{a}\in \mathfrak{C}\\ \mathfrak{a}\subseteq \mathcal{O}_K }}q^{N(\mathfrak{a})}=\sum_{n\geq 0}r_{\mathfrak{C}}(n)q^{n},
\end{equation}
where the sum is over integral ideals $\mathfrak{a}\in \mathfrak{C}$, $q=e^{2\pi i\tau}$ for $\tau\in \mathbb{H}$, and $r_{\mathfrak{C}}(n)$ is the number of integral ideals in $\mathfrak{C}$ with norm $n$. As a generalization of (\ref{Introchartheta}), the {\it ray class character theta series} associated to a character $\chi: Cl_K(\mathfrak{f})\rightarrow \mathbb{C}$ is defined by
\begin{equation}\label{Thetadec}
\Theta_{\mathfrak{f}, \chi}=\sum_{\mathfrak{C}\in Cl_K(\mathfrak{f})}\chi(\mathfrak{C})\theta_\mathfrak{C}.
\end{equation}

The study of ray class theta series stems from the classical study of {\it binary quadratic form theta series},
$$\theta_Q=\sum_{(x,y)\in E}q^{Q(x,y)}$$
where $Q(x,y)=ax^2+bxy+cy^2$ is a quadratic form with discriminant $D=b^2-4ac$ and
$$E=\{(x,y)\mid x,y\in \mathbb{Z}\}/\sim$$
for an equivalence relation $\sim$ defined in terms of solutions to $Q(x,y)=1$, which is trivial in the case that $D<0$.\\

In the case of the regular ideal class group for an imaginary
quadratic field $K$ with discriminant $D$, there is a bijection between the distinct ideal classes and definite binary quadratic forms of discriminant $D$ up to a canonical equivalence. This bijection extends to the ideals in each class, so that, for a ideal class-quadratic form pair $(\mathfrak{C}, Q)$, given an ideal $\mathfrak{a}\in \mathfrak{C}$ there is a unique, up to $K$-unit multiple, integer pair $(x,y)$ such that $N(\mathfrak{a})=Q(x,y)$. This means that 
$$w\theta_\mathfrak{C}=\theta_Q$$
where $w$ is the number of roots of unity in $K$.
These theta series are known to be modular forms of weight $1$.\\

Okano \cite{okano2021theta} extended this by showing that under certain conditions, theta series of ray classes in quadratic fields are also theta series of binary quadratic forms with congruence conditions which are modular forms of weight $1$. From \cite[Thm 2.2]{okano2021theta} and the discussion following \cite[Thm 1.2]{okano2021theta} Okano demonstrates the following result.

\begin{theorem}[Okano \cite{okano2021theta}] \label{OkanoT}
Let $K$ be a quadratic field with discriminant $D$.
    \begin{enumerate}
        \item[(i)] Suppose $K$ is an imaginary quadratic field and $\mathfrak{f}=\mathfrak{f}_0$ is a modulus with empty infinite part. Then for any $\mathfrak{C}\in Cl_K(\mathfrak{f})$, $\theta_{\mathfrak{C}}$ is a weight $1$ modular form for $\Gamma_1(-DN(\mathfrak{f}))$.

        \item[(ii)]Suppose $K$ is a real quadratic field, and $\mathfrak{f}$ a modulus where $\mathfrak{f}_\infty=\infty_1\infty_2$ includes both real embeddings. Take $\nu$ to be a totally positive integer with $\nu+1\in \mathfrak{f}_0$. Then for any $\mathfrak{C}\in Cl_K(\mathfrak{f})$, $\theta_{\mathfrak{C}}-\theta_{\mathfrak{C}[(\nu)]}$ is a weight $1$ modular form for $\Gamma_1(DN(\mathfrak{f}))$.
    \end{enumerate}
\end{theorem}

In explaining different formulas for $\sigma$ discovered by Andrews, et al. \cite{ADH}, Cohen \cite{Cohen} demonstrated an equivalence of ray class character theta series across triples of quadratic number fields.

\begin{theorem}[Cohen \cite{Cohen}] \label{CohenT}
Let $L/\mathbb{Q}$ be Galois with $G=\text{Gal}(L/\mathbb{Q})$. Suppose the center of $G$ is cyclic of index $4$. Then the subfield of $L$ fixed by its center is a biquadratic field with three quadratic subfields $K_1$, $K_2$, $K_3$. For each $i$, $L/K_i$ is abelian, so let $\mathfrak{f}_i=\mathfrak{f}_{L/K_i}$ denote the conductor satisfying the Artin-Takagi map (\ref{ATMap}). Then for each $i$, there exists two characters $\chi_{i,1}$, $\chi_{i,2}$ on $Cl_{K_i}(\mathfrak{f}_i)/N_{L/K_i}$ such that 
$$\Theta_{\mathfrak{f}_1, \chi_{1,1}}=\Theta_{\mathfrak{f}_1, \chi_{1,2}}=\Theta_{\mathfrak{f}_{2}, \chi_{2,1}}=\Theta_{\mathfrak{f}_2, \chi_{2,2}}=\Theta_{\mathfrak{f}_3, \chi_{3,1}}=\Theta_{\mathfrak{f}_3, \chi_{3,2}}.$$
\end{theorem}

In the case that $L$ is a degree $8$ extension, $\text{Gal}(L/K_i)\cong \mathbb{Z}/2\mathbb{Z}\times\mathbb{Z}/2\mathbb{Z}$ for $i=1,3$ and $\text{Gal}(L/K_2)\cong \mathbb{Z}/4\mathbb{Z}$.

\section{Proof of Theorem \ref{MainT}}

Let $L=\mathbb{Q}(\sqrt[4]{6}, i)$. Then $L$ is degree $8$ Galois extension of $\mathbb{Q}$ with Galois group $\text{Gal}(L/\mathbb{Q})\cong D_8$. Thus $L$ satisfies Condition \ref{CohenC}, and the subfield fixed by the center of $\text{Gal}(L/\mathbb{Q})$ is the biquadratic field $\mathbb{Q}(\sqrt{6},i)$, with three quadratic subfields

$K_1=\mathbb{Q}(\sqrt{-6})$, $K_2=\mathbb{Q}(i)$, and $K_3=\mathbb{Q}(\sqrt{6})$. Thus for each $i$, $L/K_i$ is a degree $4$ abelian Galois extension. See Figure \ref{MyFields} for the subfield diagram of $K$.
Applying Theorem \ref{CohenT} \cite{Cohen} yields a threefield theta series for $K_1$, $K_2$, $K_3$ which we denote by $\Theta$.
Abbreviating $\mathfrak{f}_{L/K_i}=\mathfrak{f}_i$, the conductor of $L/K_i$ which satisfies (\ref{ATMap}), and setting $\chi_i$ a character of maximal order on $Cl_{K_i}(\mathfrak{f}_i)/N_{L/K}(\mathfrak{f}_i)$ predicted by Theorem \ref{CohenT}, we have
$$\Theta=\Theta_{\mathfrak{f}_1,\chi_1}=\Theta_{\mathfrak{f}_2,\chi_2}=\Theta_{\mathfrak{f}_3,\chi_3}.$$

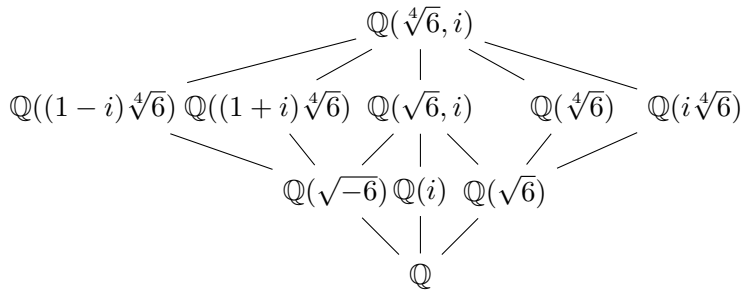
\begin{figure}
\begin{center}\begin{tikzpicture}[scale=0.75]
\node (L) {$\mathbb{Q}(\sqrt[4]{6}, i)$}
    child {node (A1) [xshift=-2.1cm] {$\mathbb{Q}((1-i)\sqrt[4]{6})$}}
    child {node (A2) [xshift=-0.9cm] {$\mathbb{Q}((1+i)\sqrt[4]{6})$}}
	child {node (B) {$\mathbb{Q}(\sqrt{6},i)$}
		child {node (k1)  {$\mathbb{Q}(\sqrt{-6})$}}
		child {node (k2) {$\mathbb{Q}(i)$}
			child {node (Q) {$\mathbb{Q}$}}
			}
		child {node (k3)  {$\mathbb{Q}(\sqrt{6})$}}}
    child {node (B1) [xshift=0.9cm] {$\mathbb{Q}(\sqrt[4]{6})$}}
    child {node (B2) [xshift=1.4cm] {$\mathbb{Q}(i\sqrt[4]{6})$}};
\draw [-] (Q) to (k1);
\draw [-] (Q) to (k3);
\draw [-] (A1) to (k1);
\draw [-] (A2) to (k1);
\draw [-] (B1) to (k3);
\draw [-] (B2) to (k3);
\end{tikzpicture}\end{center}
\caption{Subfields of $\mathbb{Q}(\sqrt[4]{6},i)$}
\label{MyFields}
\end{figure}

We will use the approach of Okano in his proof of Theorem \ref{OkanoT} to study $\Theta$ and prove Theorem \ref{MainT}. We will prove Theorem \ref{MainT} by considering the composition of $\Theta$ from (\ref{Thetadec}) for the field $K_1=\mathbb{Q}(\sqrt{-6})$. However, in Section 4 we will also consider the composition of $\Theta$ in (\ref{Thetadec}) from the perspective of $K_2=\mathbb{Q}(i)$ and $K_3=\mathbb{Q}(\sqrt{6})$ in order to compare with the $K_1$ composition, and in the case of the $K_3$ setting to compare $\rho$ with the series $\sigma$ studied in \cite{ADH} and \cite{Cohen} in relation to $K_3$.\\

Computing with Pari/GP, we calculate that the conductor of $L/K_1$ is 
$$\mathfrak{f}_1=(4\sqrt{-6})\subseteq \mathcal{O}_{K_1}=\mathbb{Z}[\sqrt{-6}],$$
with empty infinite part. Thus $K_1$, $\mathfrak{f}_1$ satisfy the hypotheses of Theorem \ref{OkanoT}, so for each ray ideal class $\mathfrak{C}\in Cl_{K_1}(\mathfrak{f}_1)$ the ray class theta function $\theta_{\mathfrak{C}}$ is a weight 1 modular form for $\Gamma_1(2304)$. Moreover, applying the $\mathbb{C}$-vector space structure on spaces of modular forms to the decomposition in (\ref{Thetadec}), we conclude that $\Theta$ is a weight $1$ modular form for $\Gamma_1(2304)$.\\

Computing using Pari/GP, we obtain that the ray class group of $\mathfrak{f}_1$ is 
$$Cl_{K_1}(\mathfrak{f}_1)\cong(\mathbb{Z}/4\mathbb{Z})^2\times \mathbb{Z}/2\mathbb{Z},$$ 
and the domain of $\chi_1$ and the Artin-Takagi map (\ref{ATMap}) is
\begin{equation}\label{RclgpQ1}
Cl_{K_1}(\mathfrak{f}_1)/N_{L/K_1}\cong \text{Gal}(L/K_1)\cong(\mathbb{Z}/2\mathbb{Z})^2.
\end{equation}
We will denote the four cosets of  $Cl_{K_1}(\mathfrak{f}_1)/N_{L/K_1}$ by $\mathcal{I}, \mathcal{J}, \mathcal{B},$ and $\mathcal{B}'$. 
Since $|Cl_{K_1}(\mathfrak{f}_1)|=32$, there are $8$ ray ideal classes in each coset, and for each class $\mathfrak{C}$ in a given coset, the value of $\chi_1(\mathfrak{C})$ is the same. A character of maximal order on $(\mathbb{Z}/2\mathbb{Z})^2$ has order $2$, so label the cosets such that 
$$\chi_1(\mathcal{I})=\chi_1(\mathcal{B})=1\text{ and } \chi_1(\mathcal{J})=\chi_1(\mathcal{B}')=-1.$$
Thus we can decompose $\Theta=\Theta_{\mathfrak{f}_1,\chi_1}$ using (\ref{Thetadec}) to obtain
\begin{equation}\label{ThetaIJBB}
\Theta=\sum_{\mathfrak{C}\in \mathcal{I}}\theta_{\mathfrak{C}}-\sum_{\mathfrak{C}\in \mathcal{J}}\theta_{\mathfrak{C}}+\sum_{\mathfrak{C}\in \mathcal{B}}\theta_{\mathfrak{C}}-\sum_{\mathfrak{C}\in \mathcal{B}'}\theta_{\mathfrak{C}}.
\end{equation}

Any ray class $\mathfrak{C}\in Cl_{K_1}(\mathfrak{f}_1)$ contains an integral ideal $\mathfrak{a}\subseteq \mathcal{O}_{K_1}=\mathbb{Z}[\sqrt{-6}]$, so $\mathfrak{C}=\mathfrak{a}P_{\mathfrak{f}_1}$. Then $\mathfrak{a}^{-1}\mathfrak{f}_1$ is a fractional ideal of $K_1$, and $[\mathfrak{a}^{-1}\mathfrak{f}_1]$ is an ideal class in the standard (narrow) class group of $K_1$. Since the class number of $K_1$ is $2$, there are $2$ ideal classes, the principal ideals, and the non-principal ideals. Define
\begin{equation}\label{beqn}
\mathfrak{b} = 
\begin{cases}
\mathbb{Z}+\mathbb{Z}\sqrt{-6}=\mathcal{O}_{K_1} \text{ if $\mathfrak{a}$ principal}\\
2\mathbb{Z}+\mathbb{Z}\sqrt{-6}=(2,\sqrt{-6}) \text{ otherwise.}
\end{cases}
\end{equation}
Then $\mathfrak{b}\subseteq \mathcal{O}_{K_1}$ is a primitive integral ideal, in the sense that it is not divisible by any non-unit rational integer. Moreover, $\mathfrak{b}$ belongs to the same standard ideal class as $\mathfrak{a}^{-1}\mathfrak{f}_1$.
So $\mathfrak{f}_1^{-1}\mathfrak{a}\mathfrak{b}$ is a non-zero principal fractional ideal. Let $z\in \mathbb{Q}(\sqrt{-6})^\times$ such that
\begin{equation}\label{ppleqn}
(z)=z\mathcal{O}_{K_1}=\mathfrak{f}_1^{-1}\mathfrak{a}\mathfrak{b}.
\end{equation}
Then $\mathfrak{a}=\mathfrak{b}^{-1}\mathfrak{f}_1(z)$.\\

Starting again with $\mathfrak{C}=[\mathfrak{a}]=\mathfrak{a}P_{\mathfrak{f}_1}$, we can write each integral element of $\mathfrak{C}$ as $\mathfrak{a}(\frac{\alpha}{\beta})$, where $\alpha,\beta\in \mathcal{O}_{K_1}$, $\beta\neq 0$, and $\alpha-\beta\in \mathfrak{f}_1$. 
To simplify, we can write each element of $\mathfrak{C}$ as $\mathfrak{a}(\alpha)$ where $\alpha\in K_1^\times$ and $\alpha=1+\frac{f}{a}$ for $f\in \mathfrak{f}_1$ and $a\in \mathcal{O}_{K_1}$ with $a$ coprime to $f$. 
 Recall the definition
$$\mathfrak{a}^{-1}=\{x\in K_1\mid x\mathfrak{a}\subseteq \mathcal{O}_{K_1}\},$$
so $\alpha\in \mathfrak{a}^{-1}$. Since $1\in \mathfrak{a}^{-1}$, we also have 
$$\alpha-1=\frac{f}{a}\in \mathfrak{a}^{-1}.$$
Furthermore, since $a$ is coprime to $f$, there exist $x,y\in \mathcal{O}_{K_1}$ with $1=ax+fy$. Then for any $s\in \mathfrak{a}$ we have $a^{-1}s=sx+sa^{-1}fy$. Here $sx\in \mathfrak{a}$ and since $a^{-1}f\in \mathfrak{a}^{-1}$ we have $sa^{-1}fy\in \mathcal{O}_{K_1}$. Thus $a^{-1}s\in \mathcal{O}_{K_1}$ for all $s\in \mathfrak{a}$ and $a^{-1}\in \mathfrak{a}^{-1}$. Likewise given $1+a^{-1}f\in 1+ \mathfrak{a}^{-1}\mathfrak{f}$ with $a^{-1}$ and $f$ coprime, it follows that $1+a^{-1}f\equiv 1\pmod{\mathfrak{f}_1}$ and $(1+a^{-1}f)\mathfrak{a}$ is integral so that $(1+a^{-1}f)\in P_{\mathfrak{f}_1}$. Thus the integral elements of $\mathfrak{C}=[\mathfrak{a}]$ are those $\alpha\mathfrak{a}$ such that $\alpha\in 1+\mathfrak{a}^{-1}\mathfrak{f}_1$.
Furthermore, $\alpha\mathfrak{a}=\alpha'\mathfrak{a}$ exactly when $\alpha'=u\alpha$ for some $u\in \mathcal{O}_{K_1}^\times$. In order for $\alpha\equiv \alpha'\equiv 1\pmod{\mathfrak{f}_1}$ it must be a $u\in \mathcal{O}_{K_1}^\times\cap (1+\mathfrak{f}_1)$. $1$ is the only such unit, so we have that each integral element of $\mathfrak{C}=[\mathfrak{a}]$ is uniquely represented by 
$$\{\mathfrak{a}(\alpha)\mid\alpha\in 1+\mathfrak{a}^{-1}\mathfrak{f}_1\}.$$
Thus (\ref{rcldefeqn}) becomes
$$\theta_\mathfrak{C}=\sum_{\alpha\in 1+\mathfrak{a}^{-1}\mathfrak{f}_1}q^{N(\mathfrak{a}(\alpha))}.$$
Using (\ref{ppleqn}) with $\mathfrak{b}$ as defined in (\ref{beqn}), we can write $\mathfrak{a}^{-1}\mathfrak{f}_1=\mathfrak{b}\left(\frac{1}{z}\right)$. Thus
$$\theta_\mathfrak{C}=\sum_{\alpha\in \mathfrak{b}(\frac{1}{z})}q^{N(\mathfrak{b}^{-1}\mathfrak{f}_1(z))}=\sum_{\alpha\in 1+\mathfrak{b}(\frac{1}{z})}q^{\frac{N(\mathfrak{f}_1)N((z\alpha))}{N(\mathfrak{b})}}. $$
We calculate $N(\mathfrak{f}_1)=96$ and set
$$N(\mathfrak{b})=a=\begin{cases} 1 & \mathfrak{a} \text{ principal}\\ 2 & \text{otherwise.}\end{cases}$$
Noting that $z\alpha\in z+\mathfrak{b}$, we can instead sum over $\beta\in z+\mathfrak{b}$ which yields
\begin{equation}\label{tsumbeta}
\theta_\mathfrak{C} = \sum_{\beta\in z+\mathfrak{b}}q^{\frac{96}{a}N(\beta)}.
\end{equation}

To simplify, we can reduce $z$ by an element of $\mathfrak{b}$ to get $x_1, y_1\in \mathbb{Q}$ in the interval $[0,1]$ such that 
$$x_1a+y_1\sqrt{-6}\in z+\mathfrak{b}.$$
Then
$$z+\mathfrak{b} = \{(x+x_1)a+(y+y_1)\sqrt{-6}\mid x,y\in \mathbb{Z}\}.$$
We can then calculate for 
$$\beta=(x+x_1)a+(y+y_1)\sqrt{-6}\in z+\mathfrak{b}$$
the exponent
\begin{equation}\label{betanormeqn}
\frac{96}{a}N(\beta)=96(a(x+x_1)^2+\tfrac{6}{a}(y+y_1)^2).
\end{equation}
Let $M$ to be the least common denominator of $x_1$ and $y_1$. Then setting $i=Mx_1$ and $j=My_1$, we have $i,j\in \mathbb{Z}$. We use (\ref{betanormeqn}) to rewrite (\ref{tsumbeta}) as a sum over pairs of integers
\begin{equation*}
\theta_\mathfrak{C} = \sum_{x,y\in \mathbb{Z}}q^{96\left(a(x+x_1)^2+\frac{6}{a}(y+y_1)^2\right)}= \sum_{\substack{x\equiv i\:(M)\\
y\equiv j\:(M)}}q^{\frac{96}{M^2}\left(ax^2+\frac{6}{a}y^2\right)}.
\end{equation*}

This process yields congruence conditions with which to write each ray class theta series as a binary quadratic form theta series . For example, one class in $\mathcal{I}$ is $[(43+14\sqrt{-6})]$. We can calculate $x_1=\frac{1}{2}$ and $y_1=\frac{5}{24}$ so that $M=24$, $i=12$, and $j=5$. As this is a class of principal ideals, we conclude that
$$\theta_{[(43+14\sqrt{-6})]}=\sum_{\substack{x\equiv 12\:(24)\\
y\equiv 5\:(24)}}q^{\frac{1}{6}x^2+y^2}.$$
Note that for $x\equiv 12\pmod{24}$ and $y\equiv 5\pmod{24}$, we have
$$\frac{1}{6}x^2+y^2\equiv 1\pmod{24}$$
so this $q$-series includes only terms $q^n$ with $n\equiv 1\pmod{24}$.\\

\begin{table}
\centering
\begin{tabular}{|c|c|c|c|c|c|c|c|}
\hline
& $\mathfrak{C}$ & $x_1$& $y_1$&$M$&$i$& $j$& $Q\;(24)$ \\
\hline
& $[\mathcal{O}_{K_1}]$& $0$& $\frac{1}{24}$& 24& 0 & 1 & 1\\
& $[(43+14\sqrt{-6})]$& $\frac{1}{2}$&$\frac{5}{24}$ & 24&12 & 5&1\\
& $[(1-2\sqrt{-6})]$ &$\frac{1}{2}$ & $\frac{1}{24}$& 24&12 & 1& 1\\
$\mathcal{I}$& $[(211-72\sqrt{-6})]$&0 & $\frac{19}{24}$& 24 &0 &19 &1\\
& $[(5, 2+\sqrt{-6})]$&$\frac{1}{8}$ &$\frac{11}{12}$ &24 &3 &22 &5\\
& $[(15125, 14912+\sqrt{-6})]$& $\frac{1}{8}$&$\frac{1}{12}$ &24 & 3& 2&5\\
& $[(125, 37+\sqrt{-6})]$& $\frac{3}{8}$&$\frac{7}{12}$ &24 &9 &14 &5\\
& $[(378125, 45162+\sqrt{-6})]$&$\frac{3}{8}$ & $\frac{5}{12}$&24 &9 &10 &5\\
\hline
& $[(13)]$& 0 &$\frac{11}{24}$ &24 &0 &11 &1\\
& $[(559+182\sqrt{-6})]$& $\frac{1}{2}$&$\frac{7}{24}$ &24 &12 & 7&1\\
& $[(13-26\sqrt{-6})]$& $\frac{1}{2}$&$\frac{13}{24}$ &24 &12 &13 &1\\
$\mathcal{J}$& $[(2743-936\sqrt{-6})]$& 0&$\frac{7}{24}$ &24 & 0& 7&1\\
& $[(65, 26+13\sqrt{-6})]$& $\frac{5}{8}$&$\frac{11}{12}$ &24 & 15& 22&5\\
& $[(196625,193856+13\sqrt{-6})]$& $\frac{5}{8}$&$\frac{1}{12}$ &24 &15 & 2&5\\
& $[(1625, 481+13\sqrt{-6})]$& $\frac{7}{8}$&$\frac{7}{12}$ &24 &21 & 14&5\\
& $[(4915625,587106+13\sqrt{-6})]$& $\frac{7}{8}$&$\frac{5}{12}$ &24 & 21& 10&5\\
\hline
& $[(91+13\sqrt{-6})]$& $\frac{3}{4}$& $\frac{19}{24}$&24 &18 & 19&7\\
& $[(2821+1833\sqrt{-6})]$& $\frac{3}{4}$& $\frac{13}{24}$ &24 & 18& 13&7\\
& $[(247-169\sqrt{-6})]$& $\frac{1}{4}$& $\frac{7}{24}$ &24 &6 &7 &7\\
$\mathcal{B}$& $[(24817-3809\sqrt{-6})]$& $\frac{1}{4}$& $\frac{1}{24}$ &24 &6 &1 &7\\
& $[(3575,806+13\sqrt{-6})]$& $\frac{3}{8}$& $\frac{1}{3}$ &24 &9 &8 &11\\
& $[(10814375,8452106+13\sqrt{-6})]$& $\frac{1}{8}$& $\frac{5}{6}$ &24 & 3& 20&11\\
& $[(89375, 50856+13\sqrt{-6})]$& $\frac{3}{4}$& $\frac{5}{6}$ &24 &9 &20 &11\\
& $[(270359375, 192296481+13\sqrt{-6})]$& $\frac{1}{8}$& $\frac{1}{3}$ &24 &3 & 8&11\\
\hline
& $[(7+\sqrt{-6})]$& $\frac{3}{4}$& $\frac{7}{24}$ &24 &18 & 7&7\\
& $[(217+141\sqrt{-6})]$& $\frac{3}{4}$& $\frac{1}{24}$ &24 & 18& 1&7\\
& $[(19-13\sqrt{-6})]$& $\frac{1}{4}$& $\frac{19}{24}$ &24 &6 & 19&7\\
$\mathcal{B}'$& $[(1909-293\sqrt{-6})]$& $\frac{1}{4}$& $\frac{13}{24}$ &24 & 6& 13&7\\
& $[(275,62+\sqrt{-6})]$& $\frac{7}{8}$& $\frac{1}{3}$ &24 & 21& 8&11\\
& $[(831875, 650162+\sqrt{-6})]$& $\frac{5}{8}$& $\frac{5}{6}$ &24 &15 & 20&11\\
& $[(6875, 3912+\sqrt{-6})]$& $\frac{7}{8}$& $\frac{5}{6}$ &24 & 21& 20&11\\
& $[(20796875, 14792037+\sqrt{-6})]$& $\frac{5}{8}$& $\frac{1}{3}$ &24 &15 & 8&11\\
\hline
\end{tabular}
\caption{Classes of $Cl_{K_1}(\mathfrak{f}_1)$ and corresponding congruence conditions.}\label{K1table}
\end{table}

Table \ref{K1table} contains the results of this calculation for all ray classes. The final column gives $Q(x,y)\pmod{24}$ for $x\equiv i\pmod{24}$ and $y\equiv j\pmod{24}$ with $Q(x,y)=\frac{1}{6}x^2+y^2$ if $\mathfrak{C}$ is a ray class of principal ideals, and $Q(x,y)=\frac{1}{2}x^2+\frac{1}{3}y^2$ if $\mathfrak{C}$ is a ray class of non-principal ideals. From this calculation we can see that ideal classes are made up entirely of ideals with norm congruent to exactly one of $1, 5, 7$ or $11$ modulo $24$. Further, classes in $\mathcal{I}$ and $\mathcal{J}$ contain ideals with norm $1$ or $5$ modulo $24$, while $\mathcal{B}$ and $\mathcal{B}'$ contain classes of ideals with norm $7$ or $11$ modulo $24$. For each $\mathfrak{C}\in \mathcal{B}$, there exists a corresponding $\mathfrak{C}'\in \mathcal{B}'$, such that $i\equiv \pm i'\pmod{24}$ and $j=\pm j'\pmod{24}$. It follows that 
$$\left\{ax^2+\frac{6}{a}y^2\middle| (x,y)\equiv (i,j)\!\!\!\!\pmod{24}\!\right\}=\left\{ax^2+\frac{6}{a}y^2\middle| (x,y)\equiv (i',j')\!\!\!\!\pmod{24}\!\right\}$$
as multisets, and as a result $\theta_{\mathfrak{C}}=\theta_{\mathfrak{C}'}$. Applying this to (\ref{ThetaIJBB}), we see that in fact
\begin{equation}\label{ThetaIJ}
\Theta = \sum_{\mathfrak{C}\in \mathcal{I}}\theta_{\mathfrak{C}}-\sum_{\mathfrak{C}\in \mathcal{J}}\theta_{\mathfrak{C}}.
\end{equation}
We can decompose $\mathcal{I}=\mathcal{I}_1\sqcup\mathcal{I}_5$ and $\mathcal{J}=\mathcal{J}_1\sqcup\mathcal{J}_5$ such that for $\mathfrak{C}\in \mathcal{I}_1\cup \mathcal{J}_1$, $\theta_\mathfrak{C}$ has $q$-exponents congruent to $1$ modulo $24$ and for $\mathfrak{C}\in \mathcal{I}_5\cup \mathcal{J}_5$, $\theta_\mathfrak{C}$ has $q$-exponents congruent to $5$ modulo $24$. Then
\begin{equation*}
\Theta=\left(\sum_{\mathfrak{C}\in \mathcal{I}_1}\theta_{\mathfrak{C}}-\sum_{\mathfrak{C}\in \mathcal{J}_1}\theta_{\mathfrak{C}}\right)+\left(\sum_{\mathfrak{C}\in \mathcal{I}_5}\theta_{\mathfrak{C}}-\sum_{\mathfrak{C}\in \mathcal{J}_5}\theta_{\mathfrak{C}}\right)  = q\rho(q^{24})+q^5\rho^*(q^{24})
\end{equation*}
for unique $\rho,\rho^*\in \mathbb{Z}[[q]]$. From Theorem \ref{OkanoT}, $\theta_{\mathfrak{C}}$ is a modular form for $\Gamma(2304)$ for each $\mathfrak{C}\in Cl_{K_1}(\mathfrak{f}_1)$, so both $q\rho(q^{24})$and $q^5\rho^*(q^{24})$ are as well.\\

The congruence conditions in Table \ref{K1table} can be grouped together into the following sets:
\begin{align*}
A_{1,+}&=\{(0,1), (0,19), (12, 1), (12, 5)\}\\
A_{1,-}&=\{(0,7), (0,11), (12, 7), (12, 13)\}\\
A_{5,+}&=\{(3,2), (3,22), (9, 10), (9, 14)\}\\
A_{5,-}&=\{(15,2), (15,22), (21, 10), (21, 14)\},
\end{align*}
and corresponding pairs of integers satisfying them, each defined by
$$S_{i,*}=\{(x,y)\mid x\equiv a\!\!\!\pmod{24} \text{ and } y\equiv b\!\!\!\pmod{24} \text{ for } (a,b)\in A_{i,*}\}$$
for $i\in \{1,5\}$ and $*\in \{+,-\}.$

As a result we can write
\addtocounter{equation}{1}
\begin{alignat*}{1}
    \Theta(q)&=\sum_{(x,y)\in S_{1,+}}q^{\frac{1}{6}x^2+y^2}-\sum_{(x,y)\in S_{1,-}}q^{\frac{1}{6}x^2+y^2}\\
    &\;\;\;\;\;+\sum_{(x,y)\in S_{5,+}}q^{\frac{1}{3}x^2+\frac{1}{2}y^2}-\sum_{(x,y)\in S_{5,-}}q^{\frac{1}{3}x^2+\frac{1}{2}y^2},\\
    \rho(q)&=\sum_{(x,y)\in S_{1,+}}q^{\frac{x^2+6y^2-6}{144}}-\sum_{(x,y)\in S_{1,-}}q^{\frac{x^2+6y^2-6}{144}},\tag{\theequation} \label{rhom6eq}\\
    \rho^*(q)&=\sum_{(x,y)\in S_{5,+}}q^{\frac{2x^2+3y^2-30}{144}}-\sum_{(x,y)\in S_{5,-}}q^{\frac{2x^2+3y^2-30}{144}}.
\end{alignat*}

With $\eta(\tau)=q^{1/24}(q;q)_\infty$ the Dedekind eta function, we wish to show that 
$$f(\tau)=\frac{\eta(48\tau)^8}{\eta(24\tau)^3\eta(96\tau)^3}$$
is also modular form for $\Gamma_1(2304)$ by applying \cite[Thms 1.64, 1.65]{ono2004web}. We can see that $\frac{1}{2}(8-3-3)=1$ is an integer. Next, we calculate
$$\sum_{\delta|2304}\delta r_\delta= 24,$$
and
$$\sum_{\delta|2304}\frac{2304}{\delta} r_\delta=24.$$
This means that $f$ is modular with respect to $\Gamma_1(2304)$. To check the growth at the cusps, we consider each divisor of $2304$. It turns out that
$$\frac{2304}{24}\sum_{\delta|2304} \frac{\gcd(d,\delta)^2r_\delta}{\gcd(d,\frac{2304}{D})d\delta}=\begin{cases}
    1 & d|2304 \text{ and } d\notin\{16,48,144\}\\
    10 & d\in \{16,48,144\}.
\end{cases}$$
Further,
$$f(\tau)=q\prod_{n\geq 1}\frac{(1-q^{48n})^8}{(1-q^{24n})^3(1-q^{96n})^3},$$
so in fact $f$ vanishes at $i\infty$. From this cusp calculation, we see that $f(\tau)\in S_1(\Gamma_1(N)).$\\

Since we have already demonstrated that $q\rho(q^{24})\in M_1(\Gamma_1(N))$, it remains to confirm that these match up to Sturm's bound \cite[Thm 6.4.7]{murty2015problems}. This is given by
$$\frac{1}{12}[SL_2\mathbb{Z}:\Gamma_1(2304)]=\frac{2304^2}{12}\prod_{p|2304}\left(1-\frac{1}{p^2}\right)=294912,$$
so we need to confirm the first $294912$ terms in the expansions of the two series match. Indeed we can directly calculate this many terms of $f(\tau)$ by multiplication, and $\rho$ either sorting ideals relatively prime to $\mathfrak{f}_1$ or using (\ref{rhom6eq}). Both series match to this many terms, so we conclude that
\begin{equation}\label{rhoetaquoteqn}
q\rho(q^{24})=\frac{\eta(48\tau)^8}{\eta(24\tau)^3\eta(96\tau)^3}.
\end{equation}

\section{Equations derived from the perspective of $\mathbb{Q}(i)$ and $\mathbb{Q}(\sqrt{6})$}

\subsection{The $\mathbb{Q}(i)$ expression}

With respect to $K_2$ we have conductor $\mathfrak{f}_2=(24)$ and $Cl_{K_2}(\mathfrak{f}_2)$ is isomorphic to $\mathbb{Z}/8\mathbb{Z}\times \mathbb{Z}/4\mathbb{Z}\times \mathbb{Z}/2\mathbb{Z}$, whereas $Cl_{K_2}(\mathfrak{f}_2)/N_{L/K_2}$ is isomorphic to $\mathbb{Z}/4\mathbb{Z}$. As a result each coset contains $16$ ideal classes. We take again $\mathcal{I}$ to be the coset of ideals mapped to $1$ and $\mathcal{J}$  the coset of ideals mapped to $-1$, noting that the contribution of ideals in $\mathcal{B}$ mapped to $i$ and in $\mathcal{B}'$ mapped to $-i$ cancel completely.\\

\begin{table}
\centering
\begin{tabular}{|c|c|c|c|c|c|c|c|}
\hline
& $\mathfrak{C}$ & $x_1$& $y_1$&$M$&$i$& $j$& $Q\;(24)$ \\
\hline
& $[\mathcal{O}_{K_2}]$& $\frac{1}{24}$& $0$& 24& 1 & 0 & 1\\
& $[(13)]$& $\frac{13}{24}$&$0$ & 24&13 & 0&1\\
& $[(3956+267i)]$ &$\frac{1}{6}$ & $\frac{7}{8}$& 24&4 & 21& 1\\
& $[(51428+3471i)]$&$\frac{1}{6}$ & $\frac{3}{8}$& 24 &4 &9 &1\\
& $[(3713+2016i)]$&$\frac{7}{24}$ &$0$ &24 &7 &0 &1\\
& $[(48269+26208i)]$& $\frac{19}{24}$&$0$ &24 & 19& 0&1\\
& $[(14150356+8966667i)]$& $\frac{4}{6}$&$\frac{1}{8}$ &24 &16 &3 &1\\
$\mathcal{I}$& $[(183954628+116566671i)]$&$\frac{1}{6}$ & $\frac{5}{8}$&24 &4 &15 &1\\
& $[(2287-3086i)]$ &$\frac{7}{24}$ & $\frac{5}{12}$& 24&7 & 10& 5\\
& $[(29731-40118i)]$& $\frac{19}{24}$&$\frac{5}{12}$ &24 &19 & 10&5\\
& $[(3586-1973i)]$& $\frac{5}{12}$&$\frac{19}{24}$ &24 &10 &19 &5\\
& $[(46618-25649i)]$& $\frac{5}{12}$&$\frac{7}{24}$ &24 & 10& 7&5\\
& $[(14713007-6847726i)]$& $\frac{23}{24}$&$\frac{1}{12}$ &24 & 23& 2&5\\
& $[(191269091-89020438i)]$& $\frac{11}{24}$&$\frac{1}{12}$ &24 &11 & 2&5\\
& $[(17292386-96373i)]$& $\frac{11}{12}$&$\frac{13}{24}$ &24 &22 & 13&5\\
& $[(224801018-1252849i)]$& $\frac{11}{12}$&$\frac{1}{24}$ &24 & 22& 1&5\\
\hline
& $[(60-11i)]$& $\frac{1}{2}$& $\frac{11}{24}$&24 &12 & 11&1\\
& $[(780-143i)]$& $\frac{1}{2}$& $\frac{23}{24}$ &24 & 12& 23&1\\
& $[(63+16i)]$& $\frac{3}{8}$& $\frac{1}{3}$ &24 &9 &8 &1\\
& $[(819+208i)]$& $\frac{7}{8}$& $\frac{1}{3}$ &24 &21 &8 &1\\
& $[(244956+80117i)]$& $\frac{1}{2}$& $\frac{5}{24}$ &24 &12 &5 &1\\
& $[(3184428+1041521i)]$& $\frac{1}{2}$& $\frac{7}{24}$ &24 & 12& 7&1\\
& $[(20166+186416i)]$& $\frac{5}{8}$& $\frac{1}{3}$ &24 &15 &8 &1\\
& $[(2621619+2423408i)]$& $\frac{1}{8}$& $\frac{1}{3}$ &24 &3 & 8&1\\
$\mathcal{J}$& $[(46-43i)]$& $\frac{1}{12}$& $\frac{19}{24}$ &24 &2 & 19&5\\
& $[(598-599i)]$& $\frac{1}{12}$& $\frac{7}{24}$ &24 & 2& 7&5\\
& $[(193457-157826i)]$& $\frac{17}{24}$& $\frac{11}{12}$ &24 &17 & 22&5\\
& $[(2514941-2051738i)]$& $\frac{5}{24}$& $\frac{11}{12}$ &24 & 5& 22&5\\
& $[(17292386-96373i)]$& $\frac{7}{12}$& $\frac{13}{24}$ &24 & 14& 13&5\\
& $[(224801018-1252849i)]$& $\frac{7}{12}$& $\frac{1}{24}$ &24 &14 & 1&5\\
& $[(1036483057-195998626i)]$& $\frac{23}{24}$& $\frac{5}{12}$ &24 & 23& 10&5\\
& $[(13474279741 -2547982138i)]$& $\frac{11}{24}$& $\frac{5}{12}$ &24 &11 & 10&5\\
\hline
\end{tabular}
\caption{Classes of $Cl_{K_2}(\mathfrak{f}_2)$ in $\mathcal{I}$ and $\mathcal{J}$ and corresponding congruence conditions.}\label{K2table}
\end{table}

Table \ref{K2table} contains the calculations of integer equivalence classes modulo $24$ for each ray ideal class in $\mathcal{I}$ and $\mathcal{J}$. For brevity, we omit the calculations for $\mathcal{B}$ and $\mathcal{B}'$.
In this case $Q(x,y)=x^2+y^2$, the unique, up to equivalence quadratic form of discriminant $-4$. Due to the symmetry of $Q$ we can simplify these congruence conditions a bit, resulting in the following sets of integers:

\begin{align*}
    &S_{1,+}=\{(x,y)|\,x\equiv 0\pmod{24}\text{ and } y\equiv 1\!\!\!\pmod{6}\}\\
    &T_{1,+}=\{(x,y)|\,x\equiv 4\pmod{24}\text{ and } y\equiv 3\!\!\!\pmod{12}\}\\
    &S_{1,-}=\{(x,y)|\,x\equiv 12\!\!\!\pmod{24}\text{ and } y\equiv 1\!\!\!\pmod{6}\}\\
    &T_{1,-}=\{(x,y)|\,x\equiv 8\pmod{24}\text{ and } y\equiv 3\!\!\!\pmod{12}\}\\
    &T_{5,+}=\left\{(x,y)\middle|
    \begin{alignedat}{2}
   \phantom{\text{ or }} & x\equiv 2\pmod{24} \text{ and }  y\equiv 1\!\!\pmod{12}\\
   \text{ or }& x\equiv 10\!\!\!\pmod{24}\text{ and } y\equiv 5\!\!\pmod{12}
    \end{alignedat}\right\}\\
    &T_{5,-}=\left\{(x,y)\middle|
    \begin{alignedat}{2}
   \phantom{\text{ or }} & x\equiv 2\pmod{24} \text{ and }  y\equiv 5\!\!\pmod{12}\\
   \text{ or }& x\equiv 10\!\!\!\pmod{24}\text{ and } y\equiv 1\!\!\pmod{12}
    \end{alignedat}\right\}.
\end{align*}

This allows us to write

\begin{alignat*}{1}
    \Theta(q)=&\sum_{(x,y)\in S_{1,+}}q^{x^2+y^2}+\sum_{(x,y)\in T_{1,+}}2q^{x^2+y^2}\\
    & -\sum_{(x,y)\in S_{1,-}}q^{x^2+y^2}-\sum_{(x,y)\in T_{1,-}}2q^{x^2+y^2}\\
    &+\sum_{(x,y)\in T_{5,+}}2q^{x^2+y^2}-\sum_{(x,y)\in T_{5,-}}2q^{x^2+ y^2},\\
    \rho(q)=&\sum_{(x,y)\in S_{1,+}}q^{\frac{x^2+y^2-1}{24}}+\sum_{(x,y)\in T_{1,+}}2q^{\frac{x^2+y^2-1}{24}}\\
    & -\sum_{(x,y)\in S_{1,-}}q^{\frac{x^2+y^2-1}{24}}-\sum_{(x,y)\in T_{1,-}}2q^{\frac{x^2+y^2-1}{24}},\\
    \rho^*(q)=&\sum_{(x,y)\in T_{5,+}}2q^{\frac{x^2+y^2-5}{24}}-\sum_{(x,y)\in T_{5,-}}2q^{\frac{x^2+y^2-5}{24}}.
\end{alignat*}

\subsection{The $\mathbb{Q}(\sqrt{6})$ expression}

For $K_3$, we have conductor $\mathfrak{f}_3=(4\sqrt{6})\infty_1\infty_2$, class group $Cl_{K_3}(\mathfrak{f}_3)\cong \mathbb{Z}/4\mathbb{Z}\times (\mathbb{Z}/2\mathbb{Z})^3$ with quotient group  $Cl_{K_3}(\mathfrak{f}_3)/N_{L/K_3}\cong (\mathbb{Z}/2\mathbb{Z})^2$.\\

Determining the congruence conditions corresponding to the class theta series stemming from a real quadratic quadratic field are a bit more complicated. To begin with, we require a totally positive integer, $\nu$ for which $\nu+1$ is in $(4\sqrt{6})$. We will take $\nu=47$. If we again take $\mathcal{I}, \mathcal{J}\in Cl_{K_3}(\mathfrak{f}_3)/N_{L/K_3}$ with $\chi_3$ and $\chi_3'$ both mapping classes in $\mathcal{I}$ to $1$ and classes in $\mathcal{J}$ to $-1$, then $[(\nu)]\in \mathcal{J}$. 
Given a class $\mathfrak{C}$, we can calculate congruence conditions for 
$\theta_{\mathfrak{C}}-\theta_{[(\nu)]\mathfrak{C}},$
so each class in $\mathcal{I}$ is associated to a class in $\mathcal{J}$, and these classes share a congruence condition. 

\begin{table}[h]
\centering
\begin{tabular}{|c|c|c|c|c|c|c|c|}
\hline
& $\mathfrak{C}$ & $x_1$& $y_1$&$M$&$i$& $j$& $Q\;(24)$ \\
\hline
& $[\mathcal{O}_{K_1}]$& $\frac{1}{8}$& $\frac{1}{12}$& 24& 3 & 2 & 1\\
& $[(4927-416\sqrt{6})]$& $\frac{7}{8}$&$\frac{5}{12}$ & 24&21 & 14&1\\
& $[(631+50\sqrt{6})]$ &$\frac{3}{8}$ & $\frac{11}{12}$& 24&9 & 22& 1\\
$\mathcal{I}$& $[(2984137-16146\sqrt{6})]$&$\frac{3}{8}$ & $\frac{7}{12}$& 24 &9 &14 &1\\
& $[(73+201\sqrt{6})]$&$\frac{3}{4}$ &$\frac{23}{24}$ &24 &18 &23 &5\\
& $[(1001-2535\sqrt{6})]$& $\frac{3}{4}$&$\frac{7}{24}$ &24 & 18& 7&5\\
& $[(106363+130481\sqrt{6})]$& $\frac{3}{4}$&$\frac{5}{24}$ &24 &18 &5 &5\\
& $[(128869+1549535\sqrt{6})]$&$\frac{1}{4}$ & $\frac{11}{24}$&24 &6 &11 &5\\
\hline
& $[(13)]$& $\frac{5}{8}$&$\frac{1}{12}$ &24 &15 &2 &1\\
& $[(379-32\sqrt{6})]$& $\frac{3}{8}$&$\frac{7}{12}$ &24 &9 & 14&1\\
& $[(18203+650\sqrt{6})]$& $\frac{1}{8}$ &$\frac{7}{12}$ &24 &3 &14 &1\\
$\mathcal{J}$& $[(229549-1242\sqrt{6})]$& $\frac{7}{8}$&$\frac{7}{12}$ &24 &21 & 14&1\\
& $[(77-195\sqrt{6})]$& $\frac{3}{4}$&$\frac{19}{24}$ &24 & 18& 19&5\\
& $[(949+2613\sqrt{6})]$& $\frac{3}{4}$&$\frac{11}{24}$ &24 &18 & 11&5\\
& $[(9913+119195\sqrt{6})]$& $\frac{1}{4}$&$\frac{23}{24}$ &24 &6 & 23&5\\
& $[(1382719+1696253\sqrt{6})]$& $\frac{3}{4}$&$\frac{17}{24}$ &24 & 18& 17&5\\
\hline
\end{tabular}
\caption{Classes of $Cl_{K_3}(\mathfrak{f}_3)$ in $\mathcal{I}$ and $\mathcal{J}$ and corresponding congruence conditions.}\label{K3table}
\end{table}

Table \ref{K3table} contains the results of the calculation, for each class in $\mathcal{I}$ and $\mathcal{J}$. In this case $Q(x,y)=\frac{1}{3}x^2-\frac{1}{2}y^2$ if $\mathfrak{C}$ is a ray class of ideals with totally positive generators, and $Q(x,y)=\frac{1}{6}x^2-y^2$ otherwise. Because these conditions give each difference of theta functions $\theta_{\mathfrak{C}}-\theta_{[(\nu)]\mathfrak{C}}$ in terms of a quadratic form and congruence conditions, it suffices to use the results for the classes in $\mathcal{I}$. This gives the following sets of congruence conditions:

\begin{align*}
    &A_{1} = \{(3,2), (9,14), (9,22), (21,14)\}\\
    &A_{5} = \{(6,11), (18,5), (18,7), (18,23)\}.
\end{align*}

Take pairs $(x,y)$ for the corresponding quadratic form, and then such pairs contribute positively or negatively depending upon $\text{sgn}(2x+y)$ for classes of ideals with totally positive generators or $\text{sgn}(x+y)$ for the remaining classes. Putting these conditions together gives the following sets of integer pairs:
\begin{align*}
    & S_{1,+} = \{(x,y)\mid x\equiv a\text{ and } y\equiv b\!\!\!\!\pmod{24}\text{ for }(a,b)\in A_1,\text{and } 2x>\sqrt{6}|y| \}\\  
    & S_{1,-} = \{(x,y)\mid x\equiv a \text{ and } y\equiv b\!\!\!\!\pmod{24}\text{ for }(a,b)\in A_1,\text{and } 2x<-\sqrt{6}|y| \}\\  
     & S_{5,+} = \{(x,y)\mid x\equiv a \text{ and } y\equiv b\!\!\!\!\pmod{24}\text{ for }(a,b)\in A_5,\text{and } x>\sqrt{6}|y| \}\\ 
      & S_{5,-} = \{(x,y)\mid x\equiv a \text{ and } y\equiv b\!\!\!\!\pmod{24}\text{ for }(a,b)\in A_5,\text{and } x<-\sqrt{6}|y| \}.
\end{align*}

Indefinite quadratic forms and real quadratic fields have an additional complication. If an indefinite form $Q(x,y)$ has a fundamental solution, that is, a pair of integers $(x_0,y_0)$ for which $Q(x_0,y_0)=1$, then any solution to $Q(x,y)=m$ yields an infinite family of equivalent solutions. Alternatively from the perspective of a number field, if there exists a non-torsion unit, then any element has infinitely many associates of the same norm. As $\mathbb{Q}(\sqrt{6})$ has non-torsion unit $5+2\sqrt{6}$, we also need to consider the orbits of the unit group, as any associate integers generate the same ideal. As such, we can define an equivalence relation on $\mathbb{Z}^2$ by 
defining $(x_1,y_1)\sim(x_2,y_2)$ if $x_1=x_2$ and $y_1=y_2$ or
\begin{equation}\label{fundsoleqrel}
x_1=5x_2+ 12 y_2\;\text{ and }\;y_1=2x_2+5y_2,
\end{equation}
and extend this to be symmetric and transitive. Then define
$$\tilde{S}_{i,*}=S_{i,*}/\sim.$$

This allows us to give the following formulas, where the choice of representative for each equivalence class of integer pairs does not effect sum:

\begin{align}
  \nonumber  \Theta(q)=&\sum_{(x,y)\in \tilde{S}_{1,+}}q^{\frac{1}{3}x^2-\frac{1}{2}y^2}-\sum_{(x,y)\in \tilde{S}_{1,-}}q^{\frac{1}{3}x^2-\frac{1}{2}y^2}\\
\nonumber    &+\sum_{(x,y)\in \tilde{S}_{5,+}}q^{\frac{1}{6}x^2-y^2}-\sum_{(x,y)\in \tilde{S}_{5,-}}q^{\frac{1}{6}x^2-y^2},\\
\label{rhocompeqn} \rho(q)=&\sum_{(x,y)\in \tilde{S}_{1,+}}q^{\frac{2x^2-3y^2-6}{144}}-\sum_{(x,y)\in \tilde{S}_{1,-}}q^{\frac{2x^2-3y^2-6}{144},}\\
\nonumber  \rho^*(q)=&\sum_{(x,y)\in \tilde{S}_{5,+}}q^{\frac{x^2-6y^2-30}{144}}-\sum_{(x,y)\in \tilde{S}_{5,-}}q^{\frac{x^2-6y^2-30}{144}}.
\end{align}

\subsection{Comparison with $\sigma$}

We can use Andrews et al.'s \cite[(3.1)]{ADH} analysis of the $\mathbb{Q}(\sqrt{6})$-integers or ideals contributing to the coefficients of $\sigma$ and $\sigma^*$ and reinterpret it get a similar expressions to those just given for $\rho$. This yields the following congruence conditions:
\begin{align*}
    &A_+=\{(1,0), (11,0), (5,2), (7,2)\}\\
    &A_-=\{(5,0), (7,0), (1,2), (11,2)\}.
\end{align*}

An interesting difference here is that the formula derived for $\sigma$ are in terms of $Q(x,y)=x^2-6y^2$ while the formula for $\sigma^*$ is in terms of $-Q(x,y)=6y^2-x^2$, rather than a distinct quadratic form. The congruence conditions in $A_{\pm}$ determine whether $q^{\pm Q(x,y)}$ is counted positively or negatively, with the same conditions being used to define both $\sigma$ and $\sigma^*$. Integrating the additional condition on the sign of $Q(x,y)$, we get the following sets of integers:
\begin{align*}
    S_{1,+} &= \{(x,y)\mid x\equiv a\!\!\!\!\pmod{12}, \;y\equiv b\!\!\!\!\pmod{4}\text{ for }(a,b)\in A_+,\; |x|>\sqrt{6}|y| \}\\
    S_{1,-} &= \{(x,y)\mid x\equiv a\!\!\!\!\pmod{12}, \;y\equiv b\!\!\!\!\pmod{4}\text{ for }(a,b)\in A_-,\; |x|>\sqrt{6}|y| \}\\
    S_{23,+} &= \{(x,y)\mid x\equiv a\!\!\!\!\pmod{12},\; y\equiv b\!\!\!\!\pmod{4}\text{ for }(a,b)\in A_+,\; |x|<\sqrt{6}|y| \}\\
    S_{23,+} &= \{(x,y)\mid x\equiv a\!\!\!\!\pmod{12}, \;y\equiv b\!\!\!\!\pmod{4}\text{ for }(a,b)\in A_-,\; |x|<\sqrt{6}|y| \}.
\end{align*}

Since there is a non-torsion unit in $\mathbb{Q}(\sqrt{6})$, we again need to restrict to just distinct integer pairs with respect to $Q(x,y)$. Taking the same definition for $\sim$ as an equivalence relation on integer pairs defined by extending (\ref{fundsoleqrel}),  
for $i\in \{1,23\}$ and $*\in \{+,-\}$, we define
$$\tilde{S}_{i,*}=S_{i,*}/\sim.$$
With all of this in hand, one interpretation (\ref{sigmaeqn}) together with other results of Andrews et al. \cite[(3.1), Thms 2, 5]{ADH} is that:

\begin{align}
  \nonumber  \Theta_{(4(3+\sqrt{6})),\chi}(q)&=\sum_{(x,y)\in \tilde{S}_{1,+}}q^{x^2-6y^2}-\sum_{(x,y)\in \tilde{S}_{1,-}}q^{x^2-6y^2}\\
  \nonumber  &+\sum_{(x,y)\in \tilde{S}_{23,+}}q^{6y^2-x^2}-\sum_{(x,y)\in \tilde{S}_{23,-}}q^{6y^2-x^2},\\
  \label{sigmabqfeqn}  \sigma(q)&=\sum_{(x,y)\in \tilde{S}_{1,+}}q^{\frac{x^2-6y^2-1}{24}}-\sum_{(x,y)\in \tilde{S}_{1,-}}q^{\frac{x^2-6y^2-1}{24}},\\
  \nonumber  \sigma^*(q)&=\sum_{(x,y)\in \tilde{S}_{23,+}}q^{\frac{6y^2-x^2+1}{24}}-\sum_{(x,y)\in \tilde{S}_{23,-}}q^{\frac{6y^2-x^2+1}{24}}.
\end{align}

From this we can directly compare (\ref{rhocompeqn}) and (\ref{sigmabqfeqn}). Both are given as sums over pairs of integers satisfying congruence conditions and inequalities, restricted to take just one pair for each set of associates in $\mathbb{Z}[\sqrt{6}]$. They differ in the quadratic form defining the exponent on $q$ in these sums, and in the particular congruence conditions. These differences mean that $\rho$ gives a modular form when taking $q\mapsto q^{24}$ and multiplying by $q$, while $\sigma$ is not itself an automorphic form under a similar transformation, but in some sense shares coefficients with a Maass form \cite{Cohen}, and both can be expressed in terms of partition generating functions.

\section{Proof of Proposition \ref{PartC}} 
Here we prove that $\rho$ can be written in terms of a partition generating function.
\begin{proof}[Proof of Proposition \ref{PartC}]
We have confirmed in (\ref{rhoetaquoteqn}) that
$$q\rho(q^{24})=\frac{\eta(q^{48})^8}{\eta(q^{24})^3\eta(q^{96})^3}.$$
Using the fact that $(q^a;q^a)_\infty(-q^a;q^a)_\infty=(q^{2a};q^{2a})_\infty$, it follows that
\begin{align*}
\rho(q)& = \frac{(-q;q)_\infty^3(q^2;q^2)_\infty^2}{(-q^2;q^2)_\infty^3} =(-q;q^2)_\infty^3(q^2;q^2)_\infty^2.
\end{align*}

Here $(-q;q^2)_\infty^3$ generates the partition function that counts the number of $3$-colored partitions of $n$ into distinct odd parts, and $(q^2;q^2)_\infty^2$ generates the partition function that counts the number of $2$-colored partitions into an even number of distinct even parts minus the number of $2$-colored partitions of $n$ into an odd number of distinct even parts. Together we can interpret
$$\rho(q)=\sum_{n\geq 0}r(n)q^n$$
by
$$r(n)=r_e(n)-r_o(n),$$
where $r_e(n)$ counts the number of partitions of $n$ into distinct $3$-colored odd parts and an even number of distinct $2$-colored even parts, and $r_o(n)$ counts the number of partitions of $n$ into distinct $3$-colored odd parts and an odd number of distinct $2$-colored even parts.
 \end{proof}

\end{document}